\documentclass[oneside,12pt]{amsart}
\usepackage{amsmath, amsfonts,amsthm,times,graphics}

 \makeatletter
\renewcommand*\subjclass[2][2010]{%
  \def\@subjclass{#2}%
  \@ifundefined{subjclassname@#1}{%
    \ClassWarning{\@classname}{Unknown edition (#1) of Mathematics
      Subject Classification; using '2010'.}%
  }{%
    \@xp\let\@xp\subjclassname\csname subjclassname@#1\endcsname
  }%
}

 \makeatother
\newtheorem{theorem}{Theorem}[section]
\newtheorem{lemma}[theorem]{Lemma}

\theoremstyle{definition}

 \makeatletter
\renewcommand*\subjclass[2][2010]{%
  \def\@subjclass{#2}%
  \@ifundefined{subjclassname@#1}{%
    \ClassWarning{\@classname}{Unknown edition (#1) of Mathematics
      Subject Classification; using '1991'.}%
  }{%
    \@xp\let\@xp\subjclassname\csname subjclassname@#1\endcsname
  }%
}
 \makeatother
\begin{document}

\textwidth137mm
\textheight225mm
\hoffset1mm
\voffset-15mm

\title[An Extension of a Congruence by Kohnen]
{An Extension of a Congruence by Kohnen}

\author{Romeo Me\v strovi\' c}
\address{Department of Mathematics,
Maritime Faculty, University of Montenegro, 
Dobrota 36, 85330 Kotor, Montenegro 
romeo@ac.me}

\begin{abstract}
Let $p>3$ be a prime, and let  $q_p(2)=(2^{p-1}-1)/p$ be 
the Fermat quotient of $p$ to base 2. 
Recently, Z. H. Sun proved that
  \begin{equation*}
\sum_{k=1}^{p-1}\frac{1}{k\cdot 2^k}\equiv q_p(2)-\frac{p}{2}q_p(2)^2
\pmod{p^2}
   \end{equation*}
which is a generalization of a congruence due to W. Kohnen.
In this note we give an elementary proof of the above
congruence which is based on 
several combinatorial identities 
and  congruences involving the Fermat quotient $q_p(2)$, 
harmonic or alternating harmonic sums. 
  \end{abstract}

  \maketitle
{\renewcommand{\thefootnote}{}\footnote{2010 {\it Mathematics Subject 
Classification.} Primary 11B75, 11A07;  Secondary 11B65,  05A10.

{\it Keywords and phrases.}  congruence, Fermat quotient,
harmonic number}
\setcounter{footnote}{0}}

\section{Introduction and Main Result}

Using a polynomial method, W. Kohnen \cite[Theorem]{k} 
proved that for any odd prime $p$,
     \begin{equation}\label{con1}
\sum_{k=1}^{p-1}\frac{1}{k\cdot 2^k}\equiv
\sum_{k=1}^{(p-1)/2}\frac{(-1)^{k-1}}{k}\pmod{p}.
    \end{equation}
Here, as usually in the sequel, we consider the congruence relation 
modulo a prime power $p^e$ extended to the ring of rational numbers
with denominators not divisible by $p$. 
For such fractions we put $m/n\equiv r/s \,(\bmod{\,\,p^e})$ 
if and only if $ms\equiv nr\,(\bmod{\,\,p^e})$, and the residue
class of $m/n$ is the residue class of $mn'$ where 
$n'$ is the inverse of $n$ modulo $p^e$.  

In the proof of the above congruence Kohnen \cite[the congruence (3) and 
the congruence after this]{k} showed that
   \begin{equation}\label{con2}
\sum_{k=1}^{p-1}\frac{1}{k\cdot 2^k}\equiv
-\frac{1}{2}\sum_{k=1}^{p-1}\frac{2^k}{k}\equiv
-\frac{1}{2}\sum_{k=1}^{p-1}\frac{(-1)^{k}}{k}\pmod{p}.
    \end{equation}
Now the congruence (\ref{con1}) immediately follows 
 from  (\ref{con2}) and the fact that the sum on 
the left  of (\ref{con2}) can be rewrite as
 $$
\sum_{k=1}^{(p-1)/2}\frac{(-1)^{k}}{k}
+\sum_{k=1}^{(p-1)/2}\frac{(-1)^{p-k}}{p-k}
\equiv 2\sum_{k=1}^{(p-1)/2}\frac{(-1)^{k}}{k}\pmod{p}.
  $$ 
We point also that Z. W. Sun proved in \cite{s2'} that 
for any odd prime $p$,
 \begin{equation}\label{con1'}
\sum_{k=1}^{(p-1)/2}\frac{1}{k\cdot 2^k}\equiv
\sum_{k=1}^{\left[3p/4\right]}\frac{(-1)^{k-1}}{k}\pmod{p},
    \end{equation}
where $[a]$ denotes the integer part of a real number $a$.

The congruence (\ref{con1'}) with the bound $\left[p/2^n\right]$, 
$n=1,2,\ldots$, instead of $(p-1)/2$ in the sum on the right 
hand side of (\ref{con1'}) was generalized by 
W. Kohnen \cite[Theorem]{k2}.

The congruences (\ref{con1}) and (\ref{con2}) may be very interesting if 
we observe their connection with the Fermat quotient.
The Fermat Little Theorem states that if $p$ is a 
prime and $a$ is an integer not divisible by $p$,
then $a^{p-1}\equiv 1\,(\bmod{\,\,p})$. This gives rise
to the definition of the {\it Fermat quotient of $p$ to base $a$},
    $$
q_p(a):=\frac{a^{p-1}-1}{p},
    $$
which is an integer. 
It is well known that divisibility of Fermat quotient $q_p(a)$ by $p$
has numerous applications which include the Fermat Last Theorem and
squarefreeness testing (see \cite{em}, \cite{gr1} and \cite{r}). 
 
A particular interesting one, due to Glaisher (\cite{gl}; also
see \cite{gr2})
for a prime $p\ge 3$, is
  \begin{equation*}
\sum_{k=1}^{p-1}\frac{2^k}{k}\equiv 
-2q_p(2)\pmod{p}.
  \end{equation*}

Recently,  Z. H. Sun  \cite{s2} established the following extension
of  the congruence (\ref{con1}).

\vspace{2mm}

\vspace{3mm} \noindent {\bf Theorem.}
(\cite[Theorem 4.1(iii)]{s2}.) {\it Let $p\ge 5$ be a prime. Then}
 \begin{equation}\label{con3}
 \sum_{k=1}^{p-1}\frac{1}{k\cdot 2^k}\equiv q_p(2)-\frac{p}{2}q_p(2)^2
\pmod{p^2}.
 \end{equation}

 Sun's proof \cite[Lemmas 4.1-4.3]{s2} of the congruence (\ref{con3})
is based on the congruential properties of Mirimanoff polynomials
obtained by "the anti-derivative method". In his proof 
it was also used the congruence for the sum 
$\sum_{k=1}^{(p-1)/2}1/k\,(\bmod{\,\, p^3})$
obtained in \cite[Theorem 5.2 (c)]{s1} 
whose proof is deduced by a standard tecnique for determining
power sums  $\sum_{k=1}^{(p-1)/2}k^r$ ($r=1,2,\ldots$) in terms 
of Bernoulli numbers. Our proof of the Theorem given in the next section
is entirely elementary and it is based 
on some combinatorial identities, numerous classical
and new  congruences involving the Fermat quotient $q_p(2)$, 
harmonic and alternating harmonic sums. 
These auxiliary congruences are interesting in themselves, such as
  \begin{equation*}
\sum_{k=1}^{p-1}\frac{(-1)^kH_{k-1}}{k}\equiv
2\sum_{1\le i<j\le p-1\atop  j\, even}\frac{1}{ij}\equiv
q_p(2)^2\pmod{p}.
   \end{equation*}
Furthermore, notice that Sun's method \cite{s2'} and 
Kohnen's method \cite{k2} may be applied 
to extend the congruence (\ref{con1'}) modulo $p^2$.
Both these congruences involve harmonic and alternating harmonic type 
sums.

\vspace{1mm}
\noindent {\bf Remarks.} 
Quite recently, Z. W. Sun 
\cite[Proof of Theorem 1.1, the congruence
after (2.3)]{s4} noticed that by a result of Z. H. Sun 
\cite[Corollary 3.3]{s2}, 
   \begin{equation}\label{con4'}
\sum_{k=1}^{(p-1)/2}\frac{(-1)^{k-1}}{k}\equiv q_p(2)
-\frac{p}{2}q_p(2)^2-(-1)^{(p+1)/2}pE_{p-3}\pmod{p^2},
    \end{equation}
where $E_n$ $(n=0,1,2,\ldots)$  are {\it Euler numbers},
that is, integers defined recursively by
   $$
E_0=1, \quad \mathrm{and}\quad\sum_{0\le k\le n\atop k\,\, even}
{n\choose k}E_{n-k} \quad \mathrm{for}\quad n=1,2,3,\ldots
   $$
(it is well known  that $E_{2n-1}=0$ for each $n=1,2,\ldots$).

Comparing (\ref{con3}) and  (\ref{con4'}), we have 
 $$
\sum_{k=1}^{(p-1)/2}\frac{(-1)^{k-1}}{k}\equiv 
\sum_{k=1}^{p-1}\frac{1}{k\cdot 2^k}-(-1)^{(p+1)/2}pE_{p-3}\pmod{p^2},
  $$
whence we conclude that  the congruence (\ref{con4'}) can
be considered as another generalization of the congruence 
(\ref{con1}). 

Notice that numerous combinatorial  congruences
recently obtained by Z. W. Sun in \cite{s8}--\cite{s7} 
and by Z. H. Sun in \cite{s2} 
contain the Euler numbers $E_{p-3}$ with a prime $p$. 
Namely, many of these congruences become "supercongruences"
if and only if $E_{p-3}\equiv 0\,(\bmod\,p)$.
Using the  congruence (\ref{con4'}), a
 computation via {\tt Mathematica 8} shows that only  
three primes less than $3\cdot 10^6$ 
 satisfy the condition $E_{p-3}\equiv 0\,(\bmod\,p)$
(such primes are 149, 241 and 2946901).
Recall that investigations of such primes 
have been recently suggested by Z. W. Sun in \cite{s4}; namely, 
in \cite[Remark 1.1]{s4} Sun found the first and the second such primes,
149 and 241, and used them to discover curious 
supercongruences (1.2)--(1.5) from Theorem 1.1 in \cite{s4} involving 
$E_{p-3}$. 
 
By statistical considerations (cf. \cite[p. 447]{cd} and \cite{mr}  
in relation to search  for Wieferich and Fibonacci-Wieferich 
and Wolstenholme primes, respectively), 
in an interval $[x,y]$, there are expected to be 
       $$
\sum_{x\le p\le y}\frac{1}{p}\approx \log\frac{\log y}{\log x}
      $$
primes satisfying $E_{p-3}\equiv 0\,(\bmod\,p)$. 
In particular, it follows that in the interval 
$[3\cdot 10^6,10^{18}]$ we can expect about $1.0221$  such primes. 
Also notice that in accordance to the above estimation, 
in the interval $[2,3\cdot 10^{6}]$ we can expect about $3.06882$
primes $p$ such that $E_{p-3}\equiv 0\,(\bmod\,p)$; as noticed previously, 
our computation shows that  all these primes are 149, 241 and 2946901.

Recall that a prime $p$ is said to be a {\it Wolstenholme prime} if it 
satisfies the congruence 
 $$
{2p-1\choose p-1} \equiv 1 \pmod{p^4},
 $$
or equivalently (cf. \cite[Corollary on page 386]{mc}; also 
see \cite{gl}) that $p$ divides the numerator of $B_{p-3}$.
The only two known such primes are 16843 and 2124679, and 
by a recent result of McIntosh and Roettger from \cite[pp. 2092--2093]{mr}, 
these primes are the only two  Wolstenholme primes less than $10^9$.
Nevertheless, by using an argument based on the prime number theorem, 
McIntosh  \cite[page 387]{mc} conjectured that there are infinitely many 
Wolstenholme primes. Since in  accordance to the 
our investigations of $E_{p-3}\equiv 0\,(\bmod\,p)$ up to $p<3\cdot 10^6$,
we can assume that the remainder modulo $p$ of $E_{p-3}$ is random.
Then applying the previous mentioned  McIntosh's argument 
we propose the following 
 
\vspace{1mm}
\noindent {\bf Conjecture.} There are infinitely many primes 
$p$ such that $E_{p-3}\equiv 0\,(\bmod\,p)$.

\section{Proof of the Theorem}

For a nonnegative integer $n$ let 
  $$
H_n=1+\frac{1}{2}+\cdots +\frac{1}{n}
  $$
be the $n$th {\it harmonic number}  (we assume that $H_0=0$).

We begin with well known result.

\begin{lemma}\label{l2.1} {\rm (\cite[Lemma 2.1]{s5}).}
 If $p$ is an odd prime, then
 \begin{equation}\label{con5}
{p-1\choose k}\equiv (-1)^k-(-1)^kpH_{k}+(-1)^kp^2
\sum_{1\le i<j\le k}\frac{1}{ij}\pmod{p^3}
   \end{equation}
for each $k=1,2,\ldots,p-1$.
\end{lemma}

\begin{proof} 
For a fixed $1\le k\le p-1$ we have
   \begin{eqnarray*}
(-1)^k{p-1\choose k}&=&\prod_{i=1}^k\left(1-\frac{p}{i}\right)
\equiv 1-\sum_{i=1}^k\frac{p}{i}+\sum_{1\le i<j\le k}\frac{p^2}{ij}\pmod{p^3}\\
&=&1-pH_k+p^2\sum_{1\le i<j\le k}\frac{1}{ij}\pmod{p^3},
   \end{eqnarray*}
which is actually  the congruence (\ref{con5}). 
\end{proof}

The following congruences are well  known 
(e.g., see \cite[Proof of Corollary 1.2]{s3}).

\begin{lemma}\label{l2.2}
 Let $p\ge 5$ be a prime. Then 
 \begin{equation}\label{con6}\begin{split}
q_p(2)&\equiv \frac 12\sum_{k=1}^{p-1}\frac{(-1)^{k-1}}{k}\equiv
-\sum_{1\le i\le p-1\atop i\, even}\frac{1}{i}=-\frac 12 H_{(p-1)/2}\\
&\equiv \sum_{1\le i\le p-1\atop i\, odd}\frac{1}{i}
\pmod{p}.
 \end{split} \end{equation}
\end{lemma}

\begin{proof} Applying the binomial formula, using the identity 
$\frac{1}{n}{n\choose k}=\frac{1}{k}{n-1\choose k-1}$
and the congruence (\ref{con5}) reduced  modulo $p$, we find that
 \begin{equation}\label{con7}\begin{split}
2q_p(2)&=\frac{2^p-2}{p}=\frac{(1+1)^p-2}{p}=
\frac{\sum_{k=1}^{p-1}{p\choose k}}{p}\\
&=\sum_{k=1}^{p-1}\frac{1}{k}{p-1\choose k-1}\equiv 
\sum_{k=1}^{p-1}\frac{(-1)^{k-1}}{k}\pmod{p}.
  \end{split}\end{equation}
By  {\it Wolstenholme's theorem} (\cite{w}; also see \cite[Theorem 1]{a}
or \cite{gr3}), if $p$ is a prime greater than 3, then the numerator of the 
fraction 
$$
H_{p-1}=1+\frac{1}{2}+\frac{1}{3}+\cdots+\frac{1}{p-1}
 $$ is 
divisible by $p^2$. This together with the  congruence (\ref{con7})
gives 
   \begin{eqnarray*}
2q_p(2) &\equiv&  \sum_{k=1}^{p-1}\frac{(-1)^{k-1}}{k}=
\sum_{k=1}^{p-1}\frac{1}{k}-2\sum_{1\le i\le p-1\atop i\, even}\frac{1}{i}\\
&\equiv& -2\sum_{1\le i\le p-1\atop i\, even}\frac{1}{i}=
-H_{(p-1)/2}\pmod{p}. 
  \end{eqnarray*}
Analogously, we obtain the third  congruence from (\ref{con6}).
\end{proof}

\begin{lemma}\label{l2.3}
 Let $p\ge 5$ be a prime. Then 
  \begin{equation}\label{con8}
\sum_{k=1}^{p-1}\frac{1}{k^2}\equiv 0\pmod{p}
  \end{equation}
and 
  \begin{equation}\label{con9}
\sum_{k=1}^{(p-1)/2}\frac{1}{k^2}\equiv 0\pmod{p}.
  \end{equation}
\end{lemma}

\begin{proof}
 By a result of Bayat \cite[Theorem 3 (ii)]{b}, for any prime $p\ge 5$
the numerator of the fraction 
$1+\frac{1}{2^2}+\frac{1}{3^2}+\cdots+\frac{1}{(p-1)^2}$ is 
divisible by $p$, which is the congruence (\ref{con8}). 

Notice that the set of all quadratic residues modulo $p$ is actually 
the set $\{1^2,2^2,\cdots, ((p-1)/2)^2\}$. 
Since $i^2\equiv (p-i)^2 \,(\bmod{\,\, p})$ for each 
$i=1,\ldots,(p-1)/2$, it follows that regarding  modulo $p$
this set coincides with the set $\{((p+1)/2)^2,((p+3)/2)^2,\cdots, 
(p-1)^2\}$, and so by the mentioned result of Bayat, we have  
   \begin{equation*}
\sum_{k=1}^{(p-1)/2}\frac{1}{k^2}
\equiv \frac 12\sum_{k=1}^{p-1}\frac{1}{k^2}
\equiv 0\pmod{p}.
   \end{equation*}
This is (\ref{con9}) and the proof is completed.
\end{proof}

\begin{lemma}\label{l2.4}
 Let $n$ be a positive integer. Then
 \begin{equation}\label{con10}
\sum_{k=1}^{2n}(-1)^kH_k=\frac{1}{2}H_{n},
   \end{equation}
    \begin{equation}\label{con11}
\sum_{k=2}^{2n}\sum_{1\le i<j\le k}\frac{(-1)^k}{ij}
=\sum_{1\le i<j\le 2n\atop  j\, even}\frac{1}{ij}
    \end{equation}
and
  \begin{equation}\label{con12}
\sum_{k=2}^{2n}\frac{(-1)^kH_{k-1}}{k}
=2\sum_{1\le i<j\le 2n\atop  j\, even}\frac{1}{ij}-
\sum_{1\le i<j\le 2n}\frac{1}{ij}.
    \end{equation}
\end{lemma}

\begin{proof} The identity (\ref{con10}) easily follows by induction on $n$,
and hence its proof may be omitted. 

In order to prove the  equality  (\ref{con11}), observe that 
for fixed $i$, $j$ with $1<j\le 2n$ the sum of all terms
on the left of (\ref{con11}) containing $1/(ij)$ is equal 
to
     $$ 
  \frac{1}{ij}\sum_{k=j}^{2n}(-1)^k=
\left\{\begin{array}{ll}
0 &\mathrm{if\,\,} j \mathrm{\,\, is\,\, odd}\\
1 &\mathrm{if\,\,} j \mathrm{\,\, is\,\, even}.
\end{array}\right.
   $$
This immediately yields (\ref{con11}).

The equality in (\ref{con12}) is satisfied as follows. 
  \begin{eqnarray*}
\sum_{k=1}^{2n}\frac{(-1)^kH_{k-1}}{k}&=&
\sum_{k=2}^{2n}\frac{(-1)^k}{k}\sum_{i=1}^{k-1}\frac{1}{i}=
\sum_{1\le i<j\le 2n\atop  j\, even}\frac{1}{ij}-
\sum_{1\le i<j\le 2n\atop  j\, odd}\frac{1}{ij}\\
&=&\sum_{1\le i<j\le 2n\atop  j\, even}\frac{1}{ij}-
\left(\sum_{1\le i<j\le 2n}\frac{1}{ij}-\sum_{1\le i<j\le 2n\atop  j\, even}
\frac{1}{ij}  \right)\\
&=&2\sum_{1\le i<j\le 2n\atop  j\, even}\frac{1}{ij}-
\sum_{1\le i<j\le 2n}\frac{1}{ij}.
  \end{eqnarray*}
This completes the  proof.
\end{proof}

\begin{lemma}\label{l2.5}
 Let $p\ge 5$ be a prime. Then 
 \begin{equation}\label{con13}
q_p(2)^2 \equiv 
 2\sum_{1\le i<j\le p-1\atop i\, even,\, j\, even}\frac{1}{ij}
\equiv  2\sum_{1\le i<j\le p-1\atop i\, odd,\, j\, odd}\frac{1}{ij}\pmod{p}.
 \end{equation}
\end{lemma}
\begin{proof} The second congruence in (\ref{con6}) from 
Lemma~\ref{l2.2} and the 
congruence (\ref{con9}) from Lemma~\ref{l2.3} immediately give 
 \begin{eqnarray*}
q_p(2)^2 &\equiv&
\left(\sum_{1\le i\le p-1\atop i\, even}\frac{1}{i}\right)^2
= 2\sum_{1\le i<j\le p-1\atop i\, even,\, j\, even}\frac{1}{ij}
+\sum_{1\le i\le p-1\atop i\, even}\frac{1}{i^2}\pmod{p}\\
&=& 2\sum_{1\le i<j\le p-1\atop i\, even,\, j\, even}\frac{1}{ij}+
\frac 14\sum_{k=1}^{(p-1)/2}\frac{1}{k^2}\equiv 
2\sum_{1\le i<j\le p-1\atop i\, even,\, j\, even}\frac{1}{ij}\pmod{p}.
 \end{eqnarray*}
Further, we have 
  \begin{equation*}
\sum_{1\le i<j\le p-1\atop i\, even,\, j\, even}\frac{1}{ij}
= \sum_{1\le j<i\le p-1\atop j\, odd,\, i\, odd}\frac{1}{(p-i)(p-j)}\\
\equiv \sum_{1\le j<i\le p-1\atop j\, odd,\, i\, odd}\frac{1}{ij}\pmod{p}. 
   \end{equation*}
The above two congruences yield (\ref{con13}).
\end{proof}

\begin{lemma}\label{l2.6}
 Let $p\ge 5$ be a prime. Then 
 \begin{equation}\label{con14}
\sum_{k=1}^{p-1}\frac{(-1)^kH_{k-1}}{k}\equiv
2\sum_{1\le i<j\le p-1\atop  j\, even}\frac{1}{ij}\equiv
q_p(2)^2\pmod{p}.
   \end{equation}
\end{lemma}

\begin{proof} Applying the fact that $p\mid H_{p-1}$ 
and the congruence (\ref{con8}) of Lemma~\ref{l2.3} to the left hand side of  
the identity
  $$
2\sum_{1\le i<j\le p-1}\frac{1}{ij}=
\left(\sum_{k=1}^{p-1}\frac{1}{k}\right)^2-
\sum_{k=1}^{p-1}\frac{1}{k^2},
  $$
we immediately obtain
    \begin{equation}\label{con15}
\sum_{1\le i<j\le p-1}\frac{1}{ij}\equiv 0\pmod{p}.
    \end{equation}
Substituting the  congruence (\ref{con15})  into the identity (\ref{con12}) 
of Lemma~\ref{l2.4} with $2n=p-1$, we obtain the first congruence from (\ref{con14}). 

Further, taking the first congruence of (\ref{con13}) from Lemma~\ref{l2.5},
we obtain
   \begin{equation}\label{con16}\begin{split}
 \sum_{1\le i<j\le p-1\atop   j\, even}\frac{1}{ij}
&=\sum_{1\le i<j\le p-1\atop  i\,even,\, j\, even}\frac{1}{ij}
+\sum_{1\le i<j\le p-1\atop  i\,odd,\, j\, even}\frac{1}{ij}\\
&\equiv q_p(2)^2 + \sum_{1\le i<j\le p-1\atop  i\,odd,\, 
j\, even}\frac{1}{ij}\pmod{p}.
   \end{split}\end{equation}
Hence, it remains to determine 
$S:=\sum_{1\le i<j\le p-1\atop  i\,odd,\, j\, even}\frac{1}{ij}$
modulo $p$.
Let
  $$
A:=\{(i,j):\, 1\le i< j\le p-1, 
i\,\,\mathrm{odd},\,\,j\,\,\mathrm{even}\}.
   $$
Then it is easily seen that the map $f:A\to \Bbb{N}^2$ defined as
 $f(i,j)=(j-i,j)$ is a bijection from $A$ to $A$, 
and thus    
 \begin{equation}\begin{split}\label{con17}
2S&=\sum_{1\le i<j\le p-1\atop  i\,odd,\, j\, even}
\left(\frac{1}{ij}+\frac{1}{(j-i)j}\right)=
\sum_{1\le i<j\le p-1\atop  i\,odd,\, j\, even}
\frac{1}{i(j-i)}\\
&\equiv - \sum_{1\le i<j\le p-1\atop  i\,odd,\, j\, even}
\frac{1}{i(p-(j-i))}:=-S'\pmod{p}.
  \end{split}\end{equation}
Observing also that 
the map $g:A\to \Bbb{N}^2$ defined as
 $g(i,j)=(i,p-(j-i))$ is also a bijection from $A$ to $A$, 
it follows that $S'= S$. Replacing this equality
into (\ref{con17}), we obtain $3S \equiv 0\,(\bmod{\,\,p})$, that is,  
  $$
S=\sum_{1\le i<j\le p-1\atop  i\,odd,\, j\, even}\frac{1}{ij}
\equiv 0 \pmod{p}.
  $$
Substituting this into (\ref{con16}), we obtain the second congruence 
of (\ref{con14}). This completes the proof.
\end{proof}

\begin{lemma}\label{l2.7}
 Let $p\ge 5$ be a prime. Then 
 \begin{equation}\label{con18}
H_{(p-1)/2}\equiv -2q_p(2)+pq_p(2)^2\pmod{p^2}.
   \end{equation}
\end{lemma}
\begin{proof}  After summation of the congruence (\ref{con5}) 
of Lemma~\ref{l2.1} over $k$,  
using the identities (\ref{con10}) and (\ref{con11}) 
from Lemma~\ref{l2.4} with $n=(p-1)/2$,   we find that  
 \begin{eqnarray*}
2^{p-1}-1&=&(1+1)^{p-1}-1=\sum_{k=1}^{p-1}{p-1\choose k}\\
&\equiv& \sum_{k=1}^{p-1}(-1)^k-p\sum_{k=1}^{p-1}(-1)^kH_{k}+
p^2\sum_{k=1}^{p-1}\sum_{1\le i<j\le k}\frac{(-1)^k}{ij}\pmod{p^3}\\
&=& -\frac{p}{2}H_{(p-1)/2}+
p^2\sum_{1\le i<j\le p-1\atop  j\, even}\frac{1}{ij}\pmod{p^3}.
  \end{eqnarray*}
Dividing the above congruence by $p$, we immediately obtain 
  \begin{equation*}
q_p(2)\equiv 
-\frac{1}{2}H_{(p-1)/2}+
p\sum_{1\le i<j\le p-1\atop  j\, even}\frac{1}{ij}\pmod{p^2},
 \end{equation*}
whence substituting the second congruence in (\ref{con14}) from 
Lemma~\ref{l2.6},  we immediately obtain (\ref{con18}).
\end{proof}

\noindent {\bf Remarks.} The congruence 
(\ref{con18}) was proved in 1938 by E. Lehmer 
\cite[the congruence (45), p. 358]{l}. This proof 
followed the method of Glaisher \cite{gl2}, which depends on Bernoulli
polynomials of fractional arguments. 
Using (\ref{con18}) and other similar congruences,
E. Lehmer obtained various criteria for the first case of Fermat Last Theorem
(cf. \cite{r}).
In the conclusion of this paper  \cite[p. 360]{l}
it was  observed that a beautiful Morley's congruence \cite{m}
published in 1895, follows immediately inserting the congruences 
(\ref{con18}) and (\ref{con9}) of Lemma~\ref{l2.3} into 
(\ref{con5}) of Lemma~\ref{l2.1} with $k=(p-1)/2$.
This congruence asserts that  for a prime $p>3$,
      $$
{p-1\choose (p-1)/2}\equiv (-1)^{(p-1)/2}4^{p-1}\pmod{p^3}.
    $$
Notice also that the congruence (\ref{con18}) reduced modulo $p$
asserts that $H_{(p-1)/2}\equiv -2q_p(2)\,(\bmod{\,p})$,
which is the congruence established in 1850 by Eisenstein \cite{e}.   
On the other hand, in 2002 T. Cai  \cite[Theorem 1]{c} generalized 
the congruence (\ref{con18}) to a congruence modulo 
a square of an arbitrary positive integer.

\begin{lemma}\label{l2.8}
 Let $p\ge 5$ be a prime. Then 
 \begin{equation}\label{con19}
\sum_{k=1}^{p-1}(-1)^kH_k^2\equiv
q_p(2)^2\pmod{p^2}
   \end{equation}
and
  \begin{equation}\label{con20}
\sum_{k=1}^{p-1}{p-1\choose k}H_k\equiv
-q_p(2)-\frac 12 pq_p(2)^2\pmod{p^2}.
   \end{equation}
\end{lemma}

\begin{proof} The identity $H_k=H_{k-1}+1/k$ gives
 \begin{equation*}\begin{split}
\sum_{k=1}^{p-1}(-1)^kH_k^2
&=\sum_{k=1}^{p-1}(-1)^k\left(H_{k-1}+\frac{1}{k}\right)^2\\
&=\sum_{k=1}^{p-1}(-1)^kH_{k-1}^2+2\sum_{k=1}^{p-1}\frac{(-1)^kH_{k-1}}{k}
+\sum_{k=1}^{p-1}\frac{(-1)^k}{k^2}\\
&=-\sum_{k=1}^{p-1}(-1)^kH_{k}^2+H_{p-1}^2+
2\sum_{k=1}^{p-1}\frac{(-1)^kH_{k-1}}{k}
+\sum_{k=1}^{p-1}\frac{(-1)^k}{k^2},
 \end{split}\end{equation*}
whence
 \begin{equation}\label{con20'}
2\sum_{k=1}^{p-1}(-1)^kH_k^2
=H_{p-1}^2+2\sum_{k=1}^{p-1}\frac{(-1)^kH_{k-1}}{k}
+\sum_{k=1}^{p-1}\frac{(-1)^k}{k^2}.
 \end{equation}
Since 
   $$
\sum_{k=1}^{p-1}\frac{(-1)^k}{k^2}=2\sum_{1\le k\le p-1\atop k\,\,
 even}\frac{1}{k^2}-
\sum_{k=1}^{p-1}\frac{1}{k^2}=\frac{1}{2}
\sum_{k=1}^{(p-1)/2}\frac{1}{k^2}-\sum_{k=1}^{p-1}\frac{1}{k^2},
  $$
taking into this (\ref{con8}) and (\ref{con9}) of Lemma~\ref{l2.3}, 
it follows that  
  \begin{equation}\label{con20''}
 \sum_{k=1}^{p-1}\frac{(-1)^k}{k^2}\equiv 0\pmod{p}.
  \end{equation}
Substituting  the congruences $H_{p-1}\equiv 0\,(\bmod{\,\, p})$, 
 (\ref{con14}) from Lemma~\ref{l2.6} and  (\ref{con20''}) into 
 (\ref{con20'}), we find that
  \begin{equation*}
2\sum_{k=1}^{p-1}(-1)^kH_k^2
=H_{p-1}^2+2\sum_{k=1}^{p-1}\frac{(-1)^kH_{k-1}}{k}
+\sum_{k=1}^{p-1}\frac{(-1)^k}{k^2}
\equiv 2q_p(2)^2\pmod{p}.
  \end{equation*}
This proves the congruence (\ref{con19}).

The congruence (\ref{con5}) from Lemma~\ref{l2.1}
reduced modulo $p^2$,  the identity (\ref{con10}) of 
Lemma~\ref{l2.4}, the congruences (\ref{con19}) 
and (\ref{con18}) of Lemma~\ref{l2.7}  
yield
  \begin{equation}\begin{split}\label{con21}
\sum_{k=1}^{p-1}{p-1\choose k}H_k
&\equiv\sum_{k=1}^{p-1}(-1)^kH_k-p\sum_{k=1}^{p-1}(-1)^kH_k^2\pmod{p^2}\\
&=\frac 12H_{(p-1)/2}-p\sum_{k=1}^{p-1}(-1)^kH_k^2\\
&\equiv \frac 12 (-2q_p(2)+pq_p(2)^2)-pq_p(2)^2\pmod{p^2}\\
&=-q_p(2)-\frac 12 pq_p(2)^2\pmod{p^2}.
    \end{split}\end{equation}
  This is the congruence (\ref{con20}) and the proof is completed.
\end{proof}

Finallly, in order to prove Theorem, 
we still need the following identity 
established in \cite[Eq. (40)]{ps} by using the {\it Sigma} 
package.

\begin{lemma}\label{l2.9}
 For a positive integer $n$ we have
 \begin{equation}\label{con22}
\sum_{k=1}^{n}{n\choose k}H_k=
2^nH_n-2^n\sum_{k=1}^{n}\frac{1}{k\cdot 2^k}.
  \end{equation}
\end{lemma}
\begin{proof}
We proceed by induction on $n\ge 1$. As (\ref{con22}) is trivially satisfied
for $n=1$, we suppose that this is also true for some $n\ge 1$. Then using
the induction hypothesis (in the last equality below), the 
 identities ${n+1\choose k}={n\choose k-1}+{n\choose k}$ 
and $H_k=H_{k-1}+1/k$ with $1\le k\le n+1$, we get
 \begin{eqnarray*}
\sum_{k=1}^{n+1}{n+1\choose k}H_k&=&\sum_{k=1}^{n+1}
\left({n\choose k-1}+{n\choose k} \right)H_k  \\
&=&\sum_{k=1}^{n+1}{n\choose k-1}\left(H_{k-1}+\frac{1}{k}\right)
+\sum_{k=1}^{n+1}{n\choose k}H_k\\
&=&\sum_{k=1}^{n+1}{n\choose k-1}H_{k-1}
+\sum_{k=1}^{n+1}\frac{1}{k}{n\choose k-1}+\sum_{k=1}^{n}{n\choose k}H_k\\
&=&2\sum_{k=1}^{n}{n\choose k}H_{k}
+\sum_{k=1}^{n+1}\frac{1}{k}{n\choose k-1}\\
&=&2^{n+1}H_n-2^{n+1}\sum_{k=1}^{n}\frac{1}{k\cdot 2^k}
+\sum_{k=0}^{n}\frac{1}{k+1}{n\choose k}.
 \end{eqnarray*}
Hence, the induction proof will be finished if we prove that
   $$
 2^{n+1}H_n-2^{n+1}\sum_{k=1}^{n}\frac{1}{k\cdot 2^k}
+\sum_{k=0}^{n}\frac{1}{k+1}{n\choose k}=
2^{n+1}H_{n+1}-2^{n+1}\sum_{k=1}^{n+1}\frac{1}{k\cdot 2^k}.
    $$
Substituting $H_{n+1}=H_n+1/(n+1)$ into above relation, 
it immediately reduces to
    $$
    \sum_{k=0}^{n}\frac{1}{k+1}{n\choose k}=2^{n+1}\left(\frac{1}{n+1}-
\frac{1}{(n+1)2^{n+1}}\right)=\frac{2^{n+1}-1}{n+1}.
   $$
The above equality is well known identity 
(see e.g., \cite[Identity 13, p. 3135]{sp})  of Lemma~\ref{l2.2},
and it can be derived by using  the binomial formula and the identity 
$\frac{1}{n+1}{n+1\choose k}=\frac{1}{k}{n\choose k-1}$ with
$1\le k\le n+1$ as follows.
  \begin{equation*}
\frac{2^{n+1}-1}{n+1}=\frac{1}{n+1}\sum_{k=1}^{n+1}{n+1\choose k}
=\sum_{k=1}^{n+1}\frac{1}{k}{n\choose k-1}=
\sum_{k=0}^{n}\frac{1}{k+1}{n\choose k}.
 \end{equation*}
Thus, the induction proof is completed.
\end{proof}

\begin{proof}[Proof of the Theorem]  
  The identity (\ref{con22}) from Lemma~\ref{l2.9} with $n=p-1$ becomes
 \begin{equation*}
2^{p-1}H_{p-1}-2^{p-1}\sum_{k=1}^{p-1}\frac{1}{k\cdot 2^k}=
\sum_{k=1}^{p-1}{p-1\choose k}H_k.
  \end{equation*}
Substituting the Wolstenholme's congruence 
$H_{p-1}\equiv 0\,(\bmod{\,\, p^2})$ and the congruence (\ref{con20}) 
of Lemma~\ref{l2.8} into above identity, we find that 
  \begin{equation*}
-2^{p-1}\sum_{k=1}^{p-1}\frac{1}{k\cdot 2^k}\equiv
-q_p(2)-\frac 12 pq_p(2)^2\pmod{p^2},
  \end{equation*}
whence we obtain
  \begin{equation*}
\sum_{k=1}^{p-1}\frac{1}{k\cdot 2^k}\equiv
\frac{q_p(2)+\frac 12 pq_p(2)^2}{2^{p-1}}=
\frac{q_p(2)+\frac 12 pq_p(2)^2}{1+pq_p(2)}\pmod{p^2},
  \end{equation*}
which in view of the fact that  
$1/(1+pq_p(2))\equiv 1-pq_p(2)\,(\bmod{\,\, p^2})$, gives 
 \begin{equation*}
\sum_{k=1}^{p-1}\frac{1}{k\cdot 2^k}\equiv
\left(q_p(2)+\frac 12 pq_p(2)^2\right)(1-pq_p(2))\equiv 
q_p(2)-\frac{p}{2} q_p(2)^2\pmod{p^2}.
  \end{equation*}
This is the desired congruence (\ref{con3}).
\end{proof}


\begin{thebibliography}{99} 
\bibitem{a} E. Alkan,  Variations on Wolstenholme's theorem,
{\it Amer. Math. Monthly} {\bf 101} (1994), 1001-1004. 

\bibitem{b} M. Bayat,  A generalization of Wolstenholme's Theorem,
{\it Amer. Math. Monthly} {\bf 104} (1997), 557--560. 


\bibitem{c} T. Cai,  A congruence involving the quotients
of Euler and its applications (I),
{\it Acta Arithmetica} {\bf 103} (2002), 313--320. 


\bibitem{cd} R. Crandall, K. Dilcher and C. Pomerance,
A search for Wieferich and Wilson primes,
{\it Math. Comp.} {\bf 66} (1997), 443--449. 

\bibitem{e}  G. Eisenstein, Eine neue Gattung zahlentheoretischer
Funktionen, welche von zwei Elementen abh\"{a}ngen und durch gewisse 
lineare Funktional-Gleichungen definiert werden,  {\it Bericht. K. Pruss. 
Akad. Wiss.} 1850, 36--42; see also G. Eisenstein, {\it Mathematische
Werke}, Vol. II, Chelsea, 1975, 705--711.    


\bibitem{em}  R. Ernvall and T. 
Mets\"{a}nkyl\"{a}, On the $p$-divisibylity of Fermat quotients,
{\it Math. Comp.} {\bf 66} (1997), 1353--1365. 


\bibitem{gl} J. W. L. Glaisher,  {\it On the residues of the sums of 
products of the first $p-1$ numbers, and their powers, to modulus $p^2$
or $p^3$},  Q.  J. Math. {\bf 31} (1900), 321--353. 


\bibitem{gl2} J. W. L. Glaisher,  On the residues of the sums of the inverse 
powers of numbers in arithmetical progression,
{\it Q.  J. Math.} {\bf 32} (1901), 271--305. 


\bibitem{gr1} A. Granville,  {\it Some conjectures related to 
Fermat's Last Theorem}, Number Theory (Banff, AB, 1988),
 de Gruyter, Berlin, 1990, 177--192.


\bibitem{gr2} A. Granville, The square of the Fermat quotient,
 {\it Integers},  {\bf 4} (2004), \# A22.

\bibitem{gr3} A. Granville,  Arithmetic properties of binomial coefficients. I.
Binomial coefficients modulo prime powers, in {\it Organic 
Mathematics--Burnaby, BC 1995},  CMS Conf. Proc., vol. 20, 
American  Mathematical Society, Providence, RI, 1997, 253-276.


\bibitem{k} W. Kohnen, A simple congruence modulo $p$,
{\it Amer. Math. Monthly} {\bf 104} (1997), 444--445. 


\bibitem{k2} W. Kohnen, Some congruences modulo primes,
{\it Monatsh.  Math.} {\bf 127} (1999), 321--324. 


\bibitem{l} E. Lehmer, On congruences 
involving Bernoulli numbers and the quotients of
Fermat and Wilson,  {\it Ann. Math.} {\bf 39} (1938), 350--360. 


\bibitem{mc} R. J. McIntosh,  On the converse of 
 Wolstenholme's theorem, {\it Acta Arith.}  {\bf 71} (1995), 381--389. 


\bibitem{mr} R. J. McIntosh and E. L. Roettger, A search for 
Fibonacci-Wieferich and Wolstenholme primes, {\it Math. Comp.} 
{\bf 76} (2007), 2087--2094.

\bibitem{m}  F. Morley, Note on the congruence 
$2^{4n}\equiv (-1)^n(2n)!/(n!)^2$, where $2n+1$ is a prime,
{\it Ann. Math.} {\bf 9} (1895), 168--170.


\bibitem{ps}  P. Paule and C. Schneider, Computer proofs
of a new family of harmonic number identities,
Adv. in Appl. Math. {\bf 31} (2003), 359--378.


\bibitem{r} P. Ribenboim, 13  {\it Lectures on Fermat's Last Theorem, 
Springer-Verlag}, New York, Heidelberg, Berlin, 1979.

\bibitem{sp} M. Z.  Spivey, Combinatorial sums and finite differences, 
{\it Discrete Math.} {\bf 307} (2007), 3130--3146.

\bibitem{s1} Z. H. Sun, Congruences concerning Bernoulli numbers
and Bernoulli polynomials,  {\it Discrete Appl. Math.} {\bf 105} (2000),
193--223.

\bibitem{s2} Z. H. Sun, Congruences involving Bernoulli and Euler numbers, 
{\it J. Number Theory}, {\bf 128} (2008), 280--312.

\bibitem{s2'} Z. W. Sun, A congruence for primes, 
{\it Proc. Amer. Math. Soc.} {\bf 123} (1995), 1341--1346.


\bibitem{s3} Z. W. Sun, Binomial coefficients, Catalan numbers 
and Lucas quotients, {\it Sci. China Math.} {\bf 53} (2010), 2473--2488;
preprint {\tt arXiv:0909.5648v11 [math.NT]} (2010).

\bibitem{s8} Z. W. Sun, On Delannoy numbers and Schr\"{o}der numbers,
{\it J. Number Theory} {\bf 131} (2011), 2387--2397;
preprint {\tt arXiv:1009.2486v4 [math.NT]} (2011).


\bibitem{s4} Z. W. Sun, Super congruences and Euler numbers,
 {\it Sci. China Math.} {\bf 54} (2011), article in press, 
preprint {\tt arXiv:1001.4453v19 [math.NT]} (2011). 


\bibitem{s6} Z. W. Sun, A refinement of the Hamme--Mortenson congruence,
preprint {\tt arXiv:1011.1902v5 [math.NT]} (2011). 


\bibitem{s7} Z. W. Sun, On congruences related to central 
binomial coefficients, {\it J. Number Theory} {\bf 131} (2011), 2219--2238;
preprint {\tt arXiv:0911.2415v16 [math.NT]} (2011).

\bibitem{s5} Z. W. Sun,  Arithmetic theory of harmonic  numbers,  
{\it Proc. Amer. Math. Soc.}, article in press, S 0002-9939(2011)10925-0;
preprint {\tt arXiv:0911.4433v6 [math.NT]} (2009).


\bibitem{w} J. W{\scriptsize olstenholme}, On certain properties
of prime numbers, {\it Quart. J. Pure Appl. Math.} {\bf 5} (1862), 35--39. 


\end{thebibliography}
\end{document}